\documentclass[12pt,english,times]{amsart}
\usepackage[T1]{fontenc}
\usepackage{amsmath,amscd}
\usepackage{amssymb}
\usepackage{amsthm}
\usepackage{yfonts}
\usepackage{mathrsfs}
\usepackage{MnSymbol}
\usepackage{pictexwd,dcpic}
\usepackage{xfrac}
\usepackage{babel}
\usepackage{enumerate}
\usepackage{hyperref}
\usepackage{anysize}
\usepackage{mdwlist} 
\usepackage{amsaddr}
\hypersetup{
    colorlinks=true,
    linkcolor=blue,
    filecolor=magenta,      
    urlcolor=cyan,}

\newtheorem{theorem}{Theorem}[section]
\newtheorem{lemma}[theorem]{Lemma}
\newtheorem{claim}[theorem]{Claim}

\newtheorem{proposition}[theorem]{Proposition}
\newtheorem{conjecture}[theorem]{Conjecture}

\theoremstyle{remark}
\newtheorem{remark}[theorem]{Remark}

\theoremstyle{definition}
\newtheorem*{definition}{Definition}
\numberwithin{equation}{section}

\newcommand{\Q}{\mathbb{Q}}

\begin{document}

\title{Schanuel Type Conjectures and Disjointness}
\author[Isaac Broudy]{Isaac A. Broudy}
\address{Department of Mathematics, University of California, Berkeley, CA, USA.}
\email{iabroudy@berkeley.edu}

\author[Sebastian Eterovi\'c]{Sebastian Eterovi\'c$^\ast$}
\address{School of Mathematics, University of Leeds, Leeds, UK.}
\email{s.eterovic@leeds.ac.uk}

\date{\today}

\thanks{$^\ast$corresponding author. Supported by NSF RTG grant DMS-1646385. ORCiD: 0000-0001-6724-5887}



\maketitle

\begin{abstract}
    Given a subfield $F$ of $\mathbb{C}$, we study the linear disjointess of the field $E$ generated by iterated exponentials of elements of $\overline{F}$, and the field $L$ generated by iterated logarithms, in the presence of Schanuel's conjecture. We also obtain similar results replacing $\exp$ by the modular $j$-function, under an appropriate version of Schanuel's conjecture, where linear disjointness is replaced by a notion coming from the action of $\mathrm{GL}_2$ on $\mathbb{C}$. We also show that for certain choices of $F$ we obtain unconditional versions of these statements. 
\end{abstract}

\section{Introduction}
Let $\exp(z)$ denote the usual complex exponential function ($\exp(z) = e^{z}$). Set $E_0=L_0=\overline{\mathbb{Q}}$ and for every positive integer $n$ define 
\begin{equation*}
    E_n := \overline{E_{n-1}(\{\exp(x)\mid x\in E_{n-1}\})}\quad\mbox{ and }\quad L_n := \overline{L_{n-1}(\{x \mid \exp(x)\in L_{n-1}\})},
\end{equation*}
and let $E := \bigcup_{n=1}^{\infty} E_n$ and $L :=\bigcup_{n=1}^{\infty} L_n$. It is shown in \cite{SchanuelConsequence} that Schanuel's conjecture implies that $E$ and $L$ are linearly disjoint over $\overline{\mathbb{Q}}$. Similar results have been obtained more recently in \cite{SemiabelianSchanuel} and \cite{Phillipon} where $\exp(z)$ is replaced by the exponential of an (semi-) abelian variety. 

The aim of this paper is to do three things. The first objective is, still assuming Schanuel's conjecture, to find more general finitely generated subfields $F$ of $\mathbb{C}$ such that if we set the initial step of the towers $E_{0}$ and $L_{0}$ to be the algebraic closure of $F$, then the resulting fields $E$ and $L$ are linearly disjoint over $\overline{F}$. This is achieved in Theorem \ref{thm:exp}.

The second objective is to show that under a version of Schanuel's conjecture for the modular $j$-function (that we call MSCD), one can produce an analogous result for the modular $j$-function. More specifically, set $J_{0} = K_{0} = \overline{F}$, where $F$ is a finitely generated subfield of $\mathbb{C}$, and inductively define
\begin{equation*}
    J_n := \overline{J_{n-1}(\{j(z)\mid z\in J_{n-1}\cap\mathbb{H}^{+}\})}\quad\mbox{ and }\quad K_n := \overline{K_{n-1}(\{z\in\mathbb{H}^{+}\mid j(z)\in K_{n-1}\})}.
\end{equation*}
Set $J := \bigcup_{n=1}^{\infty} J_n$ and $K:=\bigcup_{n=1}^{\infty} K_n$. Unlike $\exp$, the $j$-function is not a group homomorphism, and so linear disjointness is not the right notion for proving a relation between $J$ and $K$. Instead we will define a notion of disjointness that previously appeared in \cite{EterovicSchanuel} that considers the action of the group $G:=\mathrm{GL}_{2}$. This is achieved in Theorem \ref{thm:j}

The final objective is to find initial fields $F$ so that the linear disjointness of $E$ and $L$ is obtained unconditionally, that is, without having to rely on Schanuel's conjecture. We also obtain an analogous unconditional result for $j$. This is done in Theorems \ref{thm:unconditionalexp} and \ref{thm:unconditionalj}.

The methods used to prove all of our main results rely mostly on the work of \cite{Closure4j} on \emph{convenient generators}, which in turn rely heavily on the so-called Ax-Schanuel theorems: \cite{AX} in the case of $\exp$ and \cite{AxSchanuel4j} in the case of $j$. We expect that similar constructions can be performed whenever an Ax-Schanuel theorem is available in differential form, and so the methods presented here can be expected to extend to other situations, such as the exponential maps of semi-abelian varieties (using \cite{ax2} or \cite{kirby-semiab}) or the uniformization maps of Shimura varieties (using \cite{mpt}).

\subsection{Structure of the paper} In \textsection \ref{sec:prelim}, we introduce preliminary definitions and results regarding linear disjointness and $G$-disjointness. In \textsection \ref{sec:schanuel}, we review several Schanuel-type inequalities, in particular the Modular Schanuel Conjecture, and discuss the details of convenient generators. We prove the main results of this paper in \textsection \ref{sec:exp} and \textsection \ref{sec:j}.

\section{Preliminaries}
\label{sec:prelim}
\subsection{Basic notation}
\begin{itemize}
    \item If $F$ is any subfield of $\mathbb{C}$, then $\overline{F}$ its algebraic closure in $\mathbb{C}$. 
    \item If $x_{1},\ldots,x_{m}$ are elements of $\mathbb{C}$, then we use $\mathbf{x}$ to denote $(x_{1},\ldots,x_{m})$. Furthermore, if $f$ denotes a function, then we write $f(\mathbf{x})$ to mean $(f(x_{1}),\ldots,f(x_{m}))$. 
    \item As mentioned in the introduction, let $G$ denote the linear group $\mathrm{GL}_{2}$. For any subfield $F$ of $\mathbb{C}$ there is a natural action of $G(F)$ on $\mathbb{P}^{1}(\mathbb{C}) = \mathbb{C}\cup\left\{\infty\right\}$ given by M\"obius transformations as follows:
\begin{equation*}
    gx := \frac{ax+b}{cx+d},\quad\mbox{where } g=\left(\begin{matrix}
a & b\\
c & d
\end{matrix}\right),
\end{equation*} 
with $g\in G(F)$. Whenever we say that $G(F)$ \emph{acts on $\mathbb{C}$}, it will be in this manner. 
\end{itemize}

\subsection{Closures, dimensions and disjointness}

In this section we introduce the various notions of disjointess that will be used for our main results, and we review some of their properties. 

\begin{definition}
Let $F\subseteq\mathbb{C}$ be a subfield and let $A\subset \mathbb{C}$ be a finite subset.
\begin{enumerate}[(a)]
    \item Thinking of $\mathbb{C}$ as an $F$-vector space, we denote by $\mathrm{l.dim}_{F}(A)$ the dimension of the $F$-vector space $\mathrm{Span}_{F}(A)$, the $F$-linear span of $A$.
    
    Given another subset $B\subseteq\mathbb{C}$, let $\pi:\mathbb{C}\rightarrow\mathbb{C}/\mathrm{Span}_{F}(B)$ denote the quotient map. We write $\mathrm{l.dim}_{F}(A|B)$ to denote the dimension of the $F$-vector space $\pi(\mathrm{Span}_{F}(A))$. 
    \item Considering the action of $G(F)$ on $\mathbb{C}$, we denote by $\dim_{G(F)}(A)$ the number of distinct $G(F)$-orbits generated by elements of $A$. We say that $A$ is $G(F)$\emph{-independent} if $|A| = \dim_{G(F)}(A)$.
    
    Given another subset $B\subseteq\mathbb{C}$, we denote by $\dim_{G(F)}(A|B)$ the number of distinct $G(F)$-orbits generated by elements of $A$ which do not contain elements from $B$.
\end{enumerate}
\end{definition}

\begin{definition}
Let $E,F,L$ be subfields of $\mathbb{C}$ such that $E\subseteq F\cap L$. 
\begin{enumerate}[(a)]
    \item $F$ is \emph{linearly disjoint from $L$ over $E$}, denoted $F\bot^{l}_{E}L$, if every finite tuple of elements in $L$ that is $E$-linearly independent is also $F$-linearly independent. Equivalently, $F\bot^{l}_{E}L$ if and only if for any tuple $\boldsymbol{\ell}$ from $L$, $\mathrm{l.dim}_{F}\left(\boldsymbol{\ell}\right) = \mathrm{l.dim}_{E}\left(\boldsymbol{\ell}\right)$.
    \item $F$ and $L$ are $E$-\emph{free}, denoted $F\bot^{f}_{E}L$, if every finite set of elements of $L$ which is algebraically independent over $E$ is also algebraically independent over $F$. Equivalently, $F\bot^{f}_{E}L$ if and only if for any tuple $\boldsymbol{\ell}$ from $L$, $\mathrm{tr.deg.}_{F}F\left(\boldsymbol{\ell}\right) = \mathrm{tr.deg.}_{E}E\left(\boldsymbol{\ell}\right)$.
    \item We say that $F$ is \emph{$G(E)$-disjoint from $L$}, denoted $F\bot^{G}_{E}L$, if for every finite subset of elements of $L$ that is $G(E)$-independent is also $G(F)$-independent. Equivalently, $F\bot^{G}_{E}L$ if and only if for any tuple $\boldsymbol{\ell}$ from $L$, $\dim_{G(F)}\left(\boldsymbol{\ell}\right) = \dim_{G(E)}\left(\boldsymbol{\ell}\right)$. Equivalently, $F\bot^{G}_{E}L$ if and only if for any pair of elements $\ell_{1},\ell_{2}$ from $L$ for which there exists $g\in G(F)$ such that $g\ell_{1} = \ell_{2}$, there is $h\in G(E)$ such that $h\ell_{1} = \ell_{2}$.
\end{enumerate}
\end{definition}

\begin{remark}
Although the definitions are not stated in a symmetric way, both $\bot^{l}$ and $\bot^{f}$ are symmetric relations (see \cite[p. 360]{Lang} and \cite[p. 362]{Lang}). This leads naturally to the following question (which seems to be open).

\textbf{Question:} Is $G$-disjointness a symmetric relation, that is, is it true that for any three subfields $E, F, L\subseteq\mathbb{C}$ satisfying $E\subseteq F\cap L$ we have $F\bot^{G}_{E}L\iff L\bot^{G}_{E}F$?
\end{remark}

While the general case remains open, the following lemma provides a special case in which symmetry holds:
\begin{lemma} \label{LinearGL2}
Suppose that $E=L\cap F$. If $F$ is linearly disjoint from $L$ over $E$, then $F\bot^{G}_{E}L$ and $L\bot^{G}_{E}F$. 
\end{lemma}
\begin{proof}
We will show that $F\bot^{l}_{E}L$ implies $F\bot^{G}_{E}L$, and since linear disjointness is a symmetric condition, this is enough. Suppose that $\ell_1,\ell_2 \in L$ and $g\in G(F)$, such that $g\ell_1 =\ell_2$. Let $g = \begin{pmatrix}
a & b\\
c & d\\
\end{pmatrix}$, with $a,b,c,d \in F$ satisfying $ad\neq bc$. Then we have
$0 = c \ell_1\ell_2 - a\ell_1 + d\ell_2  - b$, which is a non-trivial $F$-linear combination of the elements $1, \ell_{1}, \ell_2, \ell_1\ell_2$ (all of which are in $L$). As $F$ is linearly disjoint from $L$ over $E$, there exist $\alpha,\beta,\gamma,\delta \in E$ (not all zero) preserving the linear dependence:
$$0 = \gamma \ell_1\ell_2 - \alpha \ell_1 + \delta \ell_2  - \beta.$$
We now consider the possible cases.
\begin{enumerate}[(a)]
    \item Suppose that $\alpha \ell_1 + \beta = 0$. It follows that $\ell_1\in E$, and as $g\in G(F)$, then $\ell_{2} = g\ell_{1}\in F$. This shows that $\ell_{2}\in L\cap F$. Since $\ell_{1},\ell_{2}\in E$, then there exists $h\in G(E)$ such that $h\ell_1 = \ell_2$.
    \item Suppose that $\gamma \ell_1 + \delta = 0$. Again it follows that $\ell_1\in E$, so we conclude as in (a). 
    \item Suppose that $\alpha \ell_1 + \beta \neq 0$, $\gamma \ell_1 + \delta \neq 0$, and $\alpha\delta - \beta\gamma = 0$. Then there exists $e\in E$ such that $\alpha =e \gamma$ and $\beta  =e \delta$. Then $\ell_2 = \frac{\alpha \ell_1+\beta}{\gamma \ell_1 + \delta} = e$. Thus $\ell_2 \in E$, and we can conclude with a similar argument to (a).
    \item If none of the cases (a), (b), (c) are satisfied, then the matrix $h = \begin{pmatrix}
\alpha & \beta\\
\gamma & \delta\\
\end{pmatrix}$ is in $G(E)$ and satisfies $h\ell_1 = \ell_2$. 
\end{enumerate}
\end{proof}

Finally, we will make use of the following lemma throughout \textsection \ref{sec:j}:
\begin{lemma}\label{intersection}
If $F\bot^{G}_{E}L$, then $F \cap L = E$.
\end{lemma}
\begin{proof}
Suppose $t\in L \setminus E$, and by way of contradiction suppose $t\in F$. It follows that $g= 
\begin{pmatrix}
t & 0 \\
0 & 1 
\end{pmatrix} \in G(F)$. Since $g 1 = t$ and we are assuming $F\bot^{G}_{E}L$, there exists $h\in G(E)$ such that $h 1 =t$. But $t$ is not in $E$, which is a contradiction.
\end{proof}

\subsection{The \texorpdfstring{$j$}{j}-function} 

First, we define a complex lattice, $\Lambda \subseteq \mathbb{C}$ is the additive subgroup of $\mathbb{C}$ generated by $\omega_1,\omega_2 \in \mathbb{C}$. In otherwords, $\Lambda = \mathbb{Z}\omega_1 + \mathbb{Z}\omega_2$. An elliptic curve over $\mathbb{C}$ is $E_{\Lambda} := \mathbb{C}/\Lambda$. Given two elliptic curves $E_{\Lambda}$ and $E_{\Lambda'}$, an isogeny is a nonzero analytic homomorphism mapping of the elliptic curves. There exists an isogeny between any two elliptic curves $E_{\Lambda}$ and $E_{\Lambda'}$ if and only if there exists a nonzero $z\in \mathbb{C}$ such that $z\Lambda=\Lambda'$. Hence studying the $G(\Q)$ disjointness of the lattices provides insight into the isogenies between the corresponding elliptic curves. This motivates the following definition:

\begin{definition}
Let $j:\mathbb{H}^{+}\to \mathbb{C}$ a holomorphic map given by $j(z) = \Lambda_z$, where $\Lambda_z$ is the integer lattice in $\mathbb{C}$, formed using $z$.
\end{definition}

It is well-known (see e.g. \cite[p.\~20]{masser2003heights}) that $j$ satisfies the following differential equation (and none of lower order):
\begin{equation}
    \label{eq:j}
    0 = \frac{j'''}{j'} - \frac{3}{2}\left(\frac{j''}{j'}\right)^{2} + \frac{j^{2} -1968j + 2654208}{j^{2}(j-1728)^{2}}\left(j'\right)^{2}.
\end{equation}
We recall that there is a family of polynomials $\left\{\Phi_{N}(X,Y)\right\}_{N=1}^{\infty}\subseteq\mathbb{Z}[X,Y]$ called the \emph{modular polynomials} associated with $j$ (see \cite[Chapter 5, Section 2]{lang-elliptic} for definitions and properties). Each $\Phi_{N}(X,Y)$ is irreducible in $\mathbb{C}[X,Y]$, $\Phi_{1}(X,Y) = X-Y$, and for $N\geq 2$, $\Phi_{N}(X,Y)$ is symmetric of total degree $\geq 2N$. 

We will often make use of the following fact. For every $g$ in $G(\mathbb{Q})$ we can define $\widetilde{g}$ as the unique matrix of the form $rg$ with $r\in \mathbb{Q}$ and $r>0$, so that the entries of $\widetilde{g}$ are all integers and relatively prime. Then, for every $x$ and $y$ in $\mathbb{H}$ the following statements are equivalent: 
\begin{itemize}
    \item $\Phi_{N}(j(x),j(y)) = 0$;
    \item There exists $g$ in $G$ with $gx=y$ and $\det\left(\widetilde{g}\right) = N$.
\end{itemize}

\begin{definition}
A point $z\in\mathbb{H}$ is said to be \emph{special} if there exists a non-scalar matrix $g\in G(\mathbb{Q})$, such that $gx=x$. We denote the set of all special points as $\Sigma$. 
\end{definition}

\textit{Remark:} A classical theorem of Schneider \cite{schneider} states that $\mathrm{tr.deg.}_{\Q}\Q(x,j(x))=0$ if and only if $x\in \Sigma$.

\section{Schanuel-type Conjectures}
\label{sec:schanuel}
We start by recalling the now classical conjecture of Schanuel on complex exponentiation.

\begin{conjecture}[Schanuel: SC]\label{schanuel4exp}
For every $\mathbf{x}=(x_1,\ldots,x_n)\in\mathbb{C}^n$ we have 
\begin{equation}
    \mathrm{tr.deg.}_\Q \Q(\mathbf{x},\exp(\mathbf{x}))\geq \mathrm{l.dim}_{\mathbb{Q}}(\mathbf{x}).
\end{equation}
\end{conjecture} 

We remark that SC gives an inequality statement for the transcendence degree measured over $\mathbb{Q}$. Since one of our objectives is to obtain results about linear disjointness over arbitrary fields, we need to first find a version of SC which works over a given finitely generated field. This will require the use of ``convenient generators'' for the fields, and the details will be explained in the next section. First we recall variants of SC for the $j$-function. For a detail of the origins of these variants, see \cite[\textsection 6.3]{Closure4j} and references therein.

\begin{conjecture}[Modular Schanuel with derivative: MSCD]
\label{ModularSchanuel+derivatives}
For every $z_{1},\ldots,z_{n}\in \mathbb{H}^{+}$ we have:
\begin{equation*}
    \mathrm{tr.deg.}_{\mathbb{Q}}\mathbb{Q}\left(\mathbf{z},j(\mathbf{z}),j'(\mathbf{z}),j''(\mathbf{z})\right)\geq 3\dim_{G(\mathbb{Q})}(\mathbf{z}|\Sigma).
\end{equation*}
\end{conjecture}

It is easy to see that MSCD implies the following statement without derivatives. 

\begin{conjecture}[Modular Schanuel]
\label{ModularSchanuel}
For every $z_{1},\ldots,z_{n}\in \mathbb{H}^{+}$ we have:
\begin{equation*}
    \mathrm{tr.deg.}_{\mathbb{Q}}\mathbb{Q}\left(\mathbf{z},j(\mathbf{z})\right)\geq \dim_{G(\mathbb{Q})}(\mathbf{z}|\Sigma).
\end{equation*}
\end{conjecture}

\subsection{Field derivations}

\begin{definition}
A map $\partial:\mathbb{C}\rightarrow \mathbb{C}$ is a called a \emph{derivation} if it satisfies the following two conditions:
\begin{enumerate}
\item $\partial(a+b) = \partial(a)+\partial(b)$ for every $a,b\in \mathbb{C}$.
\item $\partial(ab) = a\partial(b)+b\partial(a)$ for every $a,b\in \mathbb{C}$.
\end{enumerate}

A derivation $\partial:\mathbb{C}\rightarrow \mathbb{C}$ is called an \emph{exponential derivation} if it satisfies: 
$$\partial(\exp(z)) = \exp(z)\partial(z)$$ 
for all $z\in\mathbb{C}$. Let $\mathrm{EDer}$ denote the set of all exponential derivations. 

A derivation $\partial:\mathbb{C}\rightarrow \mathbb{C}$ is called a \emph{$j$-derivation} if it satisfies:\footnote{As shown in \cite[\textsection 5]{EterovicSchanuel}, these conditions already imply that $\partial$ will respect all the derivatives of $j$.} 
$$\partial(j(z)) = j'(z)\partial(z)\,\wedge\,\partial(j'(z)) = j''(z)\partial(z)\,\wedge\,\partial(j''(z)) = j'''(z)\partial(z),$$
for all $z\in\mathbb{H}^{+}$. Let $j\mathrm{Der}$ denote the set of all $j$-derivations. 
\end{definition}

Define 
$$C_{\exp}:=\bigcap_{\partial\in\mathrm{EDer}}\ker\partial\quad\mbox{ and }\quad C_{j}:=\bigcap_{\partial\in j\mathrm{Der}}\ker\partial.$$
Using some techniques from o-minimality, one can show that there are $|\mathbb{C}|$-many $\mathbb{C}$-linearly independent exponential derivations, and the same is true about $j$-derivations (see \cite{bays-kirby-wilkie} for the details in the case of $\exp$ and \cite[\textsection 5]{EterovicSchanuel} for the case of $j$). One can find more explicit descriptions of the sets $C_{\exp}$ and $C_{j}$ by using \emph{Khovanskii systems} of equations (see \cite[\textsection 6]{Closure4j} for $j$ and \cite[\textsection 3]{kirby-expalg} for $\exp$).

These types of derivations can also be used to define certain closure operators called \emph{pregeometries} which have associated well-defined notions of dimension (see \cite[Appendix C]{tent-ziegler} for the basic definitions and properties concerning pregeometries). 

\begin{definition}
Let $A\subseteq\mathbb{C}$ be any set. We define the set $e\mathrm{cl}(A)$ by the property: $x\in e\mathrm{cl}(A)$ if and only if $\partial(x) = 0$ for every exponential derivation $\partial$ with $A\subseteq\ker\partial$. If $A = e\mathrm{cl}(A)$, then we say that $A$ is \emph{$e\mathrm{cl}$-closed}. 

Similarly, we define the set $j\mathrm{cl}(A)$ by the property: $x\in j\mathrm{cl}(A)$ if and only if $\partial(x) = 0$ for every $j$-derivation $\partial$ with $A\subseteq\ker\partial$. If $A = j\mathrm{cl}(A)$, then we say that $A$ is \emph{$j\mathrm{cl}$-closed}. 
\end{definition}

Every $e\mathrm{cl}$-closed and every $j\mathrm{cl}$-closed subset of $\mathbb{C}$ is an algebraically closed subfield. We denote by $\dim^{e}$ the dimension associated with $e\mathrm{cl}$, and by $\dim^{j}$ the dimension associated with $j\mathrm{cl}$. For reference, $\dim^{e}$ can be defined in the following way. For any subsets $A, B\subseteq\mathbb{C}$ and for every non-negative integer $n$, $\dim^{e}(A|B) \geq n$ if and only if there exist $a_{1},\ldots,a_{n}\in e\mathrm{cl}(A)$ and $\partial_{1},\ldots,\partial_{n}\in\mathrm{EDer}$ such that $B\subseteq\ker\partial_{i}$ for $i=1,\ldots,n$ and 
\begin{equation*}
    \partial_{i}\left(a_{k}\right) = \left\{\begin{array}{cl}
         1 &, \mbox{ if } i=k\\
         0 &, \mbox{ else}
    \end{array}\right.
\end{equation*}
for every $i,k = 1,\ldots,n$. The dimension $\dim^{j}$ can be defined in an analogous way. 

\begin{lemma}[see {{\cite{kirby-expalg}}} and {{\cite{EterovicSchanuel}}}]
\label{lem:c}
$C_{\exp}$ and $C_{j}$ are countable algebraically closed subfields of $\mathbb{C}$. Furthermore,
\begin{enumerate}[(a)]
    \item For every $z\in\mathbb{C}$, $z$ is in $C_{\exp}$ if and only if $\exp(z)$ is in $C_{\exp}$. 
    \item For every $z\in\mathbb{H}^{+}$, $z$ is in $C_{j}$ if and only if $j(z)$, $j'(z)$ or $j''(z)$ is in $C_{j}$. 
\end{enumerate}
\end{lemma}

\subsection{Convenient generators}

\begin{definition}
We will say that a tuple $\mathbf{t}=(t_{1},\ldots,t_{m})$ of elements of $\mathbb{C}$ is \emph{convenient for $\exp$} if
$$\mathrm{tr.deg.}_{\mathbb{Q}}\mathbb{Q}(\mathbf{t},\exp(\mathbf{t})) = \mathrm{l.dim}_{\mathbb{Q}}(\mathbf{t}) + \dim^{e}(\mathbf{t}).$$

We will say that a tuple $\mathbf{t}=(t_{1},\ldots,t_{m})$ of elements of $\mathbb{H}^{+}$ is \emph{convenient for $j$} if
$$\mathrm{tr.deg.}_{\mathbb{Q}}\mathbb{Q}(\mathbf{t},J(\mathbf{t})) = 3\dim_{G}(\mathbf{t}|\Sigma) + \dim^{j}(\mathbf{t}).$$
\end{definition}

Convenient tuples allow us to get Schanuel-type inequalities. 

\begin{lemma} 
\label{generalizedSCHANUEL}
Suppose $\mathbf{t}\in\mathbb{C}^m$ is convenient for $\exp$. Set $F = \mathbb{Q}(\mathbf{t},\exp(\mathbf{t}))$. Then SC implies that for any $\mathbf{x}=(x_1,\ldots x_n)\in\mathbb{C}^n$ we have:
\begin{equation*}
    \mathrm{tr.deg.}_{F}F(\mathbf{x},\exp{(\mathbf{x})})\geq \mathrm{l.dim}_{\mathbb{Q}}(\mathbf{x}|\mathbf{t}).
\end{equation*}
\end{lemma} 
\begin{proof}
By SC we have that
\begin{equation*}
    \mathrm{tr.deg.}_{\mathbb{Q}}\mathbb{Q}(\mathbf{x},\mathbf{t},\exp(\mathbf{x}),\exp(\mathbf{t}))\geq\mathrm{l.dim}_{\mathbb{Q}}(\mathbf{x},\mathbf{t}).
\end{equation*}
Using the addition formula and the fact that $\mathbf{t}$ is convenient for $\exp$ as in the proof of \cite[Lemma 4.9]{gensol}, we get the result. 
\end{proof}

An analogous statement for $j$ can be found in \cite[Lemma 4.9]{gensol}. Of course, we need to address the question of whether convenient tuples exist. For the case of $j$ this was shown in \cite[Lemma 4.13]{gensol}, under the assumption of MSCD. A similar proof gives us the result for $\exp$. 

\begin{proposition}
\label{prop:generatorsexp}
Let $F\subset\mathbb{C}$ be a subfield such that $\mathrm{tr.deg.}_{\mathbb{Q}}F$ is finite. Then SC implies that there exist $\mathbf{t}=(t_{1},\ldots,t_{m})\in\mathbb{C}^m$ such that 
\begin{enumerate}
    \item[(c1):] $F\subseteq \overline{\Q(\mathbf{t},\exp{(\mathbf{t})})})$, and
    \item[(c2):] $\mathbf{t}$ is convenient for $\exp$.
\end{enumerate}
Furthermore, without loss of generality we may assume that $\mathrm{l.dim}_{\mathbb{Q}}(\mathbf{t})=m$.
\end{proposition}
\begin{proof}
By \cite[Theorem 5.6]{Closure4j} there exist $\mathbf{t}_{1}:=t_{1},\ldots,t_{k}\in\mathbb{C}\setminus C_{\exp}$ such that 
\begin{enumerate}[(a)]
    \item $K\subseteq\overline{C_{\exp}\left(\mathbf{t}_{1},\exp\left(\mathbf{t}_{1}\right)\right)}$,
    \item $\mathrm{tr.deg.}_{C_{\exp}}C_{\exp}\left(\mathbf{t}_{1},\exp\left(\mathbf{t}_{1}\right)\right) = \mathrm{l.dim}_{\mathbb{Q}}\left(\mathbf{t}_{1}|C_{\exp}\right) + \dim^{e}\left(\mathbf{t}_{1}\right)$.
\end{enumerate}
As $\mathrm{tr.deg.}_{C_{\exp}}C_{\exp}\left(\mathbf{t}_{1},\exp\left(\mathbf{t}_{1}\right)\right)$ is finite, then there is a finitely generated field $L\subseteq C$ such that $\mathrm{tr.deg.}_{C_{\exp}}C_{\exp}\left(\mathbf{t}_{1},\exp\left(\mathbf{t}_{1}\right)\right) = \mathrm{tr.deg.}_{L}L\left(\mathbf{t}_{1},\exp\left(\mathbf{t}_{1}\right)\right)$. As $F$ is finitely generated, if $M$ denotes the compositum of $L$ and $F\cap C_{\exp}$, then $M$ has finite transcendence degree over $\mathbb{Q}$. 
\begin{claim}
SC implies that there exist $\mathbf{t}_{2}=t_{k+1},\ldots,t_{m}\in\mathbb{C}\cap C_{\exp}$ such that 
\begin{enumerate}[(i)]
    \item $M\subseteq\overline{\mathbb{Q}\left(\mathbf{t}_{2},\exp\left(\mathbf{t}_{2}\right)\right)}$, and
    \item $\mathbf{t}_{2}$ is convenient for $\exp$.
\end{enumerate}
\end{claim}
\begin{proof}
If $M\subseteq\overline{\mathbb{Q}}$, then we are done. So suppose that $M$ has positive transcendence degree over $\mathbb{Q}$, and let $T$ be a transcendence basis for $M$ over $\mathbb{Q}$. As $M\subseteq C_{\exp}$, then by \cite[Theorem 1.1]{kirby-expalg}, for every $y\in T$ there are $y=y_{1},\ldots,y_{n}\in C_{\exp}$ such that they are a solution to a Khovanskii system of exponential polynomials. Although the definition of Khovanskii systems used in \cite[\textsection 3]{kirby-expalg} allows the exponential polynomials to have iterated exponentials, we appeal to \cite[Remark 6.1]{Closure4j} to ensure that no iterations of $\exp$ occur (by increasing the number of variables if necessary). Thus the Khovanskii system we obtain is made up of polynomials $p_{1},\ldots,p_n\in\mathbb{Q}[X_1,\ldots,X_n,Y_1,\ldots,Y_n]$ so that if we set $f_{i}(Z_{1},\ldots,Z_n):=p_{i}(Z_{1},\ldots,Z_n,\exp(Z_1),\ldots,\exp(Z_n))$, then
\begin{equation*}
    f_i(y_1,\ldots,y_n) =0\quad \mbox{ for all } i\in\{1,\ldots,n\},
\end{equation*}
and
\begin{equation*}
    \left|\begin{array}{ccc}
        \frac{\partial f_{1}}{\partial Z_{1}} &\cdots & \frac{\partial f_{1}}{\partial Z_{n}} \\
        \vdots & \ddots & \vdots \\
        \frac{\partial f_{n}}{\partial Z_{1}} & \cdots & \frac{\partial f_{n}}{\partial Z_{n}}
    \end{array}\right|(y_1,\ldots,y_n)\neq 0.
\end{equation*}
If we choose $n$ minimal with this property, we can guarantee that $\mathrm{l.dim}_\mathbb{Q}(y_{1},\ldots,y_n)=n$. Having this Khovanskii system guarantees that
\begin{equation*}
    \mathrm{tr.deg.}_\mathbb{Q}\mathbb{Q}(y_{1},\ldots,y_n,\exp(y_1),\ldots,\exp(y_n))\leq n,
\end{equation*}
which combined with SC guarantees that $y_{1},\ldots,y_{n}$ is convenient for $\exp$. Since $T$ is a transcendence basis, it's elements are $\mathbb{Q}$-linearly disjoint, so by repeating the above argument for every element of $T$ and combining the solutions of the various Khovanskii systems, we get the desired tuple $\mathbf{t}_{2}$. 
\end{proof}
Let $\mathbf{t}=(\mathbf{t}_{1},\mathbf{t}_{2})$. By construction, the elements of $\mathbf{t}_{1}$ are linearly disjoint with element of $\mathbf{t}_{2}$. Condition (c1) is satisfied by (a) and (i). As $L$ is contained in $M$, then condition (c2) is satisfied by (ii) and (b). 
\end{proof}

\section{Main Results for \texorpdfstring{$\exp$}{exp}}
\label{sec:exp}
Throughout this section $F$ will denote some specific choice of subfield of $\mathbb{C}$. Set $E_{0}=L_{0} = \overline{F}$, and then we define, as stated in the introduction, the towers of extensions
\begin{equation*}
    E_n := \overline{E_{n-1}(\{\exp(x)\mid x\in E_{n-1}\})}\quad\mbox{ and }\quad L_n := \overline{L_{n-1}(\{x \mid \exp(x)\in L_{n-1}\})}.
\end{equation*}
Finally we define $E := \bigcup_{n=1}^{\infty} E_n$ and $L :=\bigcup_{n=1}^{\infty} L_n$.
\begin{lemma} 
\label{genEXPlemma}
Let $F$ be any subfield of $\mathbb{C}$. For all $x\in E_{n-1}$, there exists $A \subseteq E_{n-1}$ such that $A\cup \{x\}$ is algebraic over $F(\exp(A))$. Likewise, for any $x\in L_{n-1}$ there exists $C\subseteq \mathbb{C}$ such that $\exp(C) \subseteq L_{n-1}$ then $\exp(C)\cup \{x\}$ is algebraic over $F(C)$.
\end{lemma}
\begin{proof}
Repeat the proof of \cite[Lemma]{SchanuelConsequence}, or see the proof of Lemma \ref{algebraiclemma} below.

\end{proof}

\begin{theorem}
\label{thm:exp}
Let $\mathbf{t}=(t_{1},\ldots,t_{s})\in\mathbb{C}^{s}$ be a convenient tuple for $\exp$ and set $F:=\mathbb{Q}(\mathbf{t},\exp(\mathbf{t}))$. Assume SC is true. Then $E$ is linearly disjoint from $L$ over $\overline{F}$.
\end{theorem}
\begin{proof}
With all we have done, the proof is now a small modification of the one given in \cite[Theorem]{SchanuelConsequence} with the role of SC appearing in the form of Lemma \ref{generalizedSCHANUEL}. We proceed by induction and assume that $E_{m-1}$ and $L_{n}$ are linearly disjoint over $\overline{F}$. Suppose by way of contradiction that $E_m$ and $L_n$ are not linearly disjoint over $\overline{F}$, so take a finite subset $\{l_1,\ldots,l_k\}\subseteq L_n$, which is linearly independent over $\overline{F}$, and assume that there are $\{e_1,\ldots,e_k\} \subseteq E_m$ such that $\sum_{i=1}^k l_ie_i = 0$, where at least one $e_i \neq 0$. By Lemma \ref{genEXPlemma} there exists a finite set $A\subseteq E_{m-1}$ such that $A\cup \{e_i\}_{i=1}^k$ is algebraic over $F(\exp(A))$, and a finite set $C \subseteq L_n$ such that $\exp(C)\cup \{l_i\}_{i=1}^k$ is algebraic over $F(C)$.

Now take $B\subseteq A$ and $D\subseteq C$ such that $\exp(B)$ is a transcendence basis for $F(\exp(A))$ over $F$ and $D$ is a transcendence basis for $F(C)$ over $F$. We first show that $\mathrm{l.dim}_{\mathbb{Q}}(B\cup D|\mathbf{t}) = |B| + |D|$. To this end, consider an expression of the form
\begin{equation}
    \label{eq:lineardependence}
    \sum_{b\in B} p_b b + \sum_{d\in D} q_d d + \sum_{i=1}^{s}r_{i}t_{i} =0
\end{equation}
with $p_b,q_d,r_{i}\in \mathbb{Z}$. Observe that $\sum_{b\in B} p_b b\in E_{m-1}$, $\sum_{d\in D} q_d d \in L_{n}$ and $\sum_{i=1}^{s}r_{i}t_{i}\in \overline{F}$, so (\ref{eq:lineardependence}) shows that $\sum_{b\in B} p_b b\in L_{n}$ and $\sum_{d\in D} q_d d\in E_{m-1}$. By the induction hypothesis, $E_{m-1}$ and $L_n$ are linearly disjoint over $\overline{F}$, so $E_{m-1}\cap L_n = \overline{F}$. Since $D$ is a transcendence basis over $F$, the condition $\sum_{d\in D} q_d d\in\overline{F}$ implies that the coefficients $q_d=0$ for each $d \in D$. Hence $\displaystyle\sum_{b\in B} p_b b  + \sum_{i=1}^{s}r_{i}t_{i} = 0$ and thus $\displaystyle\prod_{b\in B}(\exp{b})^{p_b} \in\overline{F}$. But $\exp{B}$ is a transcendence basis, so $p_{b}=0$ for every $b\in B$. This proves that $\mathrm{l.dim}_{\mathbb{Q}}(B\cup D|\mathbf{t}) = |B| + |D|$.

By Lemma \ref{generalizedSCHANUEL} we have $\mathrm{tr.deg.}_{F}F(B,D,\exp(B),\exp(D))\geq |B|+|D|$. We also have that: 
\begin{align*}
    \mathrm{tr.deg.}_{F}F(B,D,\exp(B),\exp(D))
    &=\mathrm{tr.deg.}_{F}F(D,\exp(B))\\
    &\leq |B|+|D|.
\end{align*}
Therefore $\mathrm{tr.deg.}_{F}F(D,\exp(B))=|B|+|D|$, so $\overline{F(\exp(B))}$ and $\overline{F(D)}$ are $\overline{F}$-free, and hence they are linearly disjoint over $\overline{F}$. Since $\{e_1,\ldots,e_k\}\subset\overline{F(\exp(B))}$ and $\{l_1,\ldots,l_k\}\subset\overline{F(D)}$, we reach a contradiction. 
\end{proof}

\subsection{Unconditional result}

Let $t_{1},\ldots,t_{s}\in\mathbb{C}\setminus C_{\exp}$ satisfy
$$\mathrm{tr.deg.}_{C_{\exp}}C_{\exp}(\mathbf{t},\exp(\mathbf{t})) = \mathrm{l.dim}_{\mathbb{Q}}(\mathbf{t}|C_{\exp})+\dim^{e}(\mathbf{t}).$$ 
As we explained in the proof of Proposition \ref{prop:generatorsexp}, the existence of tuples $\mathbf{t}$ satisfying the above equation is given by \cite[Theorem 5.6]{Closure4j}. Set $F:=C_{\exp}(\mathbf{t},\exp(\mathbf{t}))$ and define $E$ and $L$ accordingly.

\begin{lemma} 
\label{generalizedAXSCHANUEL}
Then for any $\mathbf{x}=(x_1,\ldots x_n)\in\mathbb{C}^n$ we have:
\begin{equation*}
    \mathrm{tr.deg.}_{F}F(\mathbf{x},\exp{(\mathbf{x})})\geq \mathrm{l.dim}_{\mathbb{Q}}(\mathbf{x}|\mathbf{t}\cup C_{\exp}) + \dim^{e}(\mathbf{x}|\mathbf{t}).
\end{equation*}
\end{lemma} 
\begin{proof}
By \cite[Corollary 5.2]{kirby-expalg} (a consequence of a theorem of Ax \cite[Theorem 3]{AX}) we have that
\begin{equation*}
    \mathrm{tr.deg.}_{C_{\exp}}C_{\exp}(\mathbf{x},\mathbf{t},\exp(\mathbf{x}),\exp(\mathbf{t}))\geq\mathrm{l.dim}_{\mathbb{Q}}(\mathbf{x},\mathbf{t}|C_{\exp}) + \dim^{e}(\mathbf{x},\mathbf{t}|C_{\exp}).
\end{equation*}
Using the addition formula and the fact that $\mathbf{t}$ is convenient for $\exp$ as in the proof of \cite[Lemma 5.2]{gensol}, we obtain the desired result. 
\end{proof}

We can now prove the following unconditional version of Theorem \ref{thm:exp}. 

\begin{theorem}
\label{thm:unconditionalexp}
With $F$ as above, $E\bot_{\overline{F}}^{l}L$. 
\end{theorem}
\begin{proof}
The proof is nearly the same as the proof of Theorem \ref{thm:exp}. In what follows we will only focus on the step that require extra attention. Instead of assuming SC, we will use Lemma \ref{generalizedAXSCHANUEL}. 

As before, we assume that $E_{m-1}\bot_{\overline{F}}^{l}L_{n}$ and that $E_m$ and $L_n$ are not linearly disjoint over $\overline{F}$. Suppose the set $\{l_1,\ldots,l_k\}\subseteq L_n$ is linearly independent over $\overline{F}$, and that there are $\{e_1,\ldots,e_k\} \subseteq E_m$ such that $\sum_{i=1}^k l_ie_i = 0$, where some $e_i \neq 0$. Choose $A\subseteq E_{m-1}$ and $C \subseteq L_n$ using Lemma \ref{genEXPlemma} just as before.

Take $B\subseteq A$ and $D\subseteq C$ as above. We show that $\mathrm{l.dim}_{\mathbb{Q}}(B\cup D|\mathbf{t}\cup  C_{\exp}) = |B| + |D|$. This time we need to consider an expression of the form
\begin{equation}
    \label{eq:lineardependence2}
    \sum_{b\in B} p_b b + \sum_{d\in D} q_d d + \sum_{i=1}^{s}r_{i}t_{i} =\gamma,
\end{equation}
for some $\gamma\in C_{\exp}$, with $p_b,q_d,r_{i}\in \mathbb{Z}$. We have that $\sum_{b\in B} p_b b\in E_{m-1}$, $\sum_{d\in D} q_d d \in L_{n}$ and $\gamma-\sum_{i=1}^{s}r_{i}t_{i}\in \overline{F}$, so (\ref{eq:lineardependence2}) shows that $\sum_{b\in B} p_b b\in L_{n}$ and $\sum_{d\in D} q_d d\in E_{m-1}$. Use Lemma \ref{lem:c}, apply the induction hypothesis and finish as before.
\end{proof}

\section{Main Results for \texorpdfstring{$j$}{j}}
\label{sec:j}
Throughout this section $F$ will a subfield of $\mathbb{C}$. As in the previous section, we start with $E_0 = L_0 = \overline{F}$ and then proceed inductively as follows:
\begin{equation*}
    J_n := \overline{J_{n-1}(\{j(z)\mid z\in J_{n-1}\cap\mathbb{H}^{+}\})}\quad\mbox{ and }\quad K_n := \overline{K_{n-1}(\{z\in\mathbb{H}^{+}\mid j(z)\in K_{n-1}\})},
\end{equation*}
and set $J := \bigcup_{n=1}^{\infty} J_n$ and $K:=\bigcup_{n=1}^{\infty} K_n$. We will keep this notation for the rest of the section.

\begin{remark}
\label{rem:transbasis}
Consider $\mathbb{C}$ as a degree two extension of the field of real numbers $\mathbb{R}$ (and not as an abstract field). Let $L\subseteq\mathbb{C}$ be an algebraically closed subfield. Then for every $z\in\mathbb{C}$ we have that $z\in L$ if and only if the real and imaginary parts of $z$ are in $L$. Indeed, choose $z\in L$ and write $z=a+ib$. Let $\partial:\mathbb{C}\to\mathbb{C}$ be any derivation satisfying that $L = \ker \partial$ (which exists since $L$ is algebraically closed). Using \cite[\textsection 4]{Wilkie} we know that there are derivations $\lambda,\mu:\mathbb{R}\to\mathbb{R}$ such that:
\begin{equation*}
    0 = \partial(z) = \lambda(a) - \mu(b)+i(\lambda(a)+\mu(b)),
\end{equation*}
which gives that $\lambda(a) = \mu(b) = 0$. So $\partial(a) = \lambda(a)+i\lambda(a) = 0$ and similarly $\partial(b)=0$. Therefore $a,b\in L$.
\end{remark}

\begin{lemma} \label{algebraiclemma}
For any $x\in J_{n}$, there exists a finite set $T\subseteq J_{n-1}\cap\mathbb{H}^{+}$ such that $T\cup\{x\}$ is algebraic over $F(j(T))$.

Likewise, for all $x\in K_{n}$, there exists a finite set $R\subseteq \mathbb{H}^{+}$ such that $j(R) \cup\{x\}$ is algebraic over $F(R)$ and for every $z\in R$, $j(z)\in K_{n-1}$.
\end{lemma}
\begin{proof}
By Remark \ref{rem:transbasis}, given $x\in J_{n}$ there is a finite set $S_{n-1}\subset J_{n-1}\cap\mathbb{H}^{+}$ (possibly empty) such that the set $T_{n-1}= \{j(z)\mid z\in S_{n-1}\}$ is contained in $\mathbb{H}^{+}$ and $x$ is algebraic over $J_{n-1}(T_{n-1})$.  Similarly, given $0\leq i\leq n$ and a finite set $T_{n-i}\subset J_{n-i}$, there is a finite set $S_{n-i-1}\subset J_{n-i}\cap\mathbb{H}^{+}$ such that the set $T_{n-i-1}= \{j(z)\mid z\in S_{n-i-1}\}$ is contained in $\mathbb{H}^{+}$ and  $T_{n-1}$ is algebraic over $J_{n-i-1}(T_{n-i-1})$. So we can proceed inductively to obtain finite sets $T_{0},\ldots,T_{n-1}$ such that if we set $T = \bigcup_{m<n} T_m$, then $T\cup\{x\}$ is algebraic over $F(j(T))$.

We now seek to prove the existence of the finite set $R$ with the desired properties. Using Remark \ref{rem:transbasis}, given $x\in K_{n}$ there is a finite set $R_{n-1}\subset \mathbb{H}^{+}$ such that for every $z\in R_{n-1}$, $j(z)\in K_{n-1}$ and $x$ is algebraic over $K_{n-1}(R_{n-1})$. Proceeding in analogous way to the previous paragraph, we are done. 
\end{proof}

\begin{theorem} 
\label{thm:j}
Suppose that $\mathbf{t}=(t_{1},\ldots,t_{s})\in\left(\mathbb{H}^{+}\right)^s$ is a convenient tuple for $j$. Set 
$$F:=\mathbb{Q}(\mathbf{t},j(\mathbf{t}),j'(\mathbf{t}),j''(\mathbf{t})).$$ 
Assume MSCD is true. Then $J\bot_{\overline{F}}^{G}K$ and $K\bot_{\overline{F}}^{G} J$.
\end{theorem}
\begin{proof}
We follow the same overall strategy as in the proof of Theorem \ref{thm:exp}. We show that $J\bot_{\overline{F}}^{G}K$ (the proof of $K\bot_{\overline{F}}^{G} J$ is the same with the roles of $J$ and $K$ reversed). 

We proceed by induction, assuming $J_{m-1}$ and $K_n$ are $G\left(\overline{F}\right)$-disjoint. We wish to prove that $J_m$ and $K_n$ are $G\left(\overline{F}\right)$-disjoint. Proceed by contradiction. Suppose $k_{1},k_{2}\in K_n$ are in different $G\left(\overline{F}\right)$-orbits, but that there is $g\in G(J_{m})$ such that $gk_{1}=k_{2}$. Let $a,b,c,d\in J_m$ be the coefficients of $g$ (in the usual way). Then by Lemma \ref{algebraiclemma}, we have that there exists a finite set $T\subseteq J_{n-1}\cap\mathbb{H}^{+}$ such that all the elements of $T \cup \left\{a,b,c,d\right\}$ are algebraic over $F(j(T))$. Similarly, there exists a finite set $R\subseteq \mathbb{H}^{+}$ such that $j(R) \cup \{ k_{1},k_{2}\}$ is algebraic over $F(R)$ and for every $z\in R$, $j(z)\in K_{n-1}$.

Let $P\subseteq T$ be a subset such that the set $j(P)$ is a transcendence basis for $F(j(T))$ over $F$, and let $Q\subseteq R$ be a 
transcendence basis for $F(R)$ over $F$. The definitions of $P$ and $Q$ immediately imply that $(P\cup Q)\cap \Sigma= \emptyset$. We will show that $\dim_{G(\mathbb{Q})}(P\cup Q|\mathbf{t}) = |P|+|Q|$. Choose two elements $x,y\in P\cup Q$. We consider the following cases.

\begin{enumerate}[(a)]
    \item Suppose that $x,y \in P$. If there is $g\in G(\mathbb{Q})$ such that $gx=y$, then there exists a modular polynomial $\Phi_{N}(X,Y)$, such that $\Phi_{N}(j(x),j(y))=0$, which shows that $j(x)$ and $j(y)$ are algebraically dependent over $\mathbb{Q}$. But this contradicts that $j(P)$ is a transcendence basis. 
    \item Suppose that $x,y \in Q$. If these elements were in the same $G(\mathbb{Q})$-orbit, this would contradict that $Q$ is a transcendence basis. 
    \item Suppose that $x\in P$ and $y\in Q$. If there is $g\in G(\mathbb{Q})$ such that $gx=y$, this implies that $x$ and $y$ are in $J_{m-1} \cap K_n$. By Lemma \ref{intersection} this implies that $x,y\in \overline{F}$, but that contradicts that $Q$ is a transcendence basis over $F$.
\end{enumerate}
This shows that $\dim_{G(\mathbb{Q})}(P\cup Q) = |P|+|Q|$. Now suppose that there is $x \in P \cup Q$ and $g\in G(\mathbb{Q})$ such that $gx=t_i$ for some $i\in\{1,\ldots,s\}$. This implies that $x$ is not transcendental over $F$, so $x\notin Q$. But then it must be that $x\in P$, and as $gx=t_i$, the set $\left\{x,j(x)\right\}$ is algebraic over $\mathbb{Q}\left(t_{i},j(t_{i})\right)$. This contradicts that $j(P)$ is a transcendence basis over $F$. Therefore $\dim_{G(\mathbb{Q})}(P\cup Q|\mathbf{t}) = |P|+|Q|$. Then by \cite[Lemma 4.9]{gensol} (which assumes MSCD):
$$\mathrm{tr.deg.}_{F}F(P,j(P),Q,j(Q)) \geq |P| + |Q|.$$ 
We also have that 
\begin{eqnarray*}
    \mathrm{tr.deg.}_{F}F(P,j(P),Q,j(Q)) &=& \mathrm{tr.deg.}_{F}F(P,R,j(T),j(Q))\\
    &=& \mathrm{tr.deg.}_{F}F(R,j(T))\\ 
    & =& \mathrm{tr.deg.}_{F}F(Q,j(P)),
\end{eqnarray*}
thus
\begin{equation*}
\mathrm{tr.deg.}_{F}F(P,j(P),Q,j(Q)) \leq |P|+|Q|.
\end{equation*}
Hence $F(j(P))$ and $F(Q)$ are $\overline{F}$-free, so $\overline{F(j(P))}$ and $\overline{F(Q)}$ are $\overline{F}$-free. Freedom over $\overline{F}$ implies that the fields are linearly disjoint over $\overline{F}$. Applying Lemma \ref{LinearGL2}, it follows that $\overline{F(j(P))}$ and $\overline{F(Q)}$ are $G\left(\overline{F}\right)$-disjoint, which is a contradiction.
\end{proof}

In \cite{SchanuelConsequence} it is shown that as a consequence of the main theorem, one can show (among other things) that the numbers $\pi, \log(\pi), \log(\log(\pi)),...$ are $E$-linearly independent, where $E$ is constructed as in \textsection\ref{sec:exp} with $E_0 = \overline{\mathbb{Q}}$. However,  obtaining a similar result about the $j$ function and $\pi$ is not expected. As shown in \cite[Remark 6.21]{Closure4j}, the generalized period conjecture of Grothendieck--Andr\'e which implies MSCD, also implies that $\pi\notin C_{j}$. Choosing $F=\overline{\mathbb{Q}}$ will give that $J,K\subseteq C_j$ (see \cite[Proposition 6.6]{Closure4j}), so this prevents $\pi$ from being in either $J$ or $K$. Of course, one can find elements in $C_j$ which would make the statement true in place of $\pi$ (as well as the other corollaries of \cite{SchanuelConsequence}), but since these numbers are not so natural we have decided not to pursue this.

\subsection{Unconditional result}
 
Suppose that $t_{1},\ldots,t_{s}\in\mathbb{H}^{+}\setminus C_{j}$ satisfy:
$$\mathrm{tr.deg.}_{C_{j}}C_j(\mathbf{t},j(\mathbf{t})) = \dim_{G(\mathbb{Q})}(\mathbf{t}|C_{j}) + \dim^{j}(\mathbf{t}).$$
Set $F:=C_j(\mathbf{t},j(\mathbf{t}),j'(\mathbf{t}),j''(\mathbf{t}))$, and define $J$ and $K$ as before with $J_{0} = K_{0} = F$. 
 
\begin{theorem}
\label{thm:unconditionalj}
With $F$ as above, $J\bot_{\overline{F}}^{G}K$ and $K\bot_{\overline{F}}^{G}J$.
\end{theorem}
\begin{proof}
We only focus on the differences with the proof of Theorem \ref{thm:j}. We assume $J_{m-1}$ and $K_n$ are $G\left(\overline{F}\right)$-disjoint and  suppose $k_{1},k_{2}\in K_n$ are in different $G\left(\overline{F}\right)$-orbits, but that there is $g\in G(J_{m})$ such that $gk_{1}=k_{2}$. Let $a,b,c,d\in J_m$ be the coefficients of $g$. Choose $T\subseteq J_{n-1}\cap\mathbb{H}^{+}$ and $R\subseteq \mathbb{H}^{+}$ using Lemma \ref{algebraiclemma} as before.

Let $P\subseteq T$ be a such that $j(P)$ is a transcendence basis for $F(j(T))$ over $F$, and let $Q\subseteq R$ be a 
transcendence basis for $F(R)$ over $F$. We will show that $\dim_{G(\mathbb{Q})}(P\cup Q|\mathbf{t}\cup C_j) = |P|+|Q|$. The same proof from Theorem \ref{thm:j} shows that $\dim_{G(\mathbb{Q})}(P\cup Q|\mathbf{t}) = |P|+|Q|$. Now suppose that there is $x\in P\cup Q$ and $g\in G(\mathbb{Q})$ such that $gx\in C_j$. If $x\in Q$, that contradicts that $Q$ is a transcendence basis over $F$. So then we would have that $x\in P$, but since $x\in C_j$ impies $j(x)\in C_j$ (by Lemma \ref{lem:c}), this contradicts that $j(P)$ is a transcendence basis over $F$. So $\dim_{G(\mathbb{Q})}(P\cup Q|\mathbf{t}\cup C_j) = |P|+|Q|$.

Then by \cite[Lemma 5.2]{Closure4j}:
$$\mathrm{tr.deg.}_{F}F(P,j(P),Q,j(Q)) \geq |P| + |Q|.$$ 
The rest of the proof stays the same.
\end{proof}




\bibliographystyle{alpha}
\bibliography{modularbib}{}

\newcommand{\etalchar}[1]{$^{#1}$}
\begin{thebibliography}{CDH{\etalchar{+}}09}

\bibitem[AEK22]{Closure4j}
Vahagn Aslanyan, Sebastian Eterovi\'c, and Jonathan Kirby.
\newblock A closure operator respecting the modular $j$-function, 2022.
\newblock To appear in Israel J. Math. Preprint: \url{https://rdcu.be/cWYWk}.

\bibitem[Ax71]{AX}
James Ax.
\newblock On {S}chanuel's conjectures.
\newblock {\em Ann. of Math. (2)}, 93:252--268, 1971.

\bibitem[Ax72]{ax2}
James Ax.
\newblock Some topics in differential algebraic geometry i: Analytic subgroups
  of algebraic groups.
\newblock {\em American Journal of Mathematics}, 94(4):1195--1204, 1972.

\bibitem[BKW10]{bays-kirby-wilkie}
Martin Bays, Jonathan Kirby, and A.~J. Wilkie.
\newblock A {S}chanuel property for exponentially transcendental powers.
\newblock {\em Bull. Lond. Math. Soc.}, 42(5):917--922, 2010.

\bibitem[BPSS22]{SemiabelianSchanuel}
Cristiana Bertolin, Patrice Philippon, Biswajyoti Saha, and Ekata Saha.
\newblock Semi-abelian analogues of {S}chanuel conjecture and applications.
\newblock {\em J. Algebra}, 596:250--288, 2022.

\bibitem[CDH{\etalchar{+}}09]{SchanuelConsequence}
Chuangxun Cheng, Brian Dietel, Mathilde Herblot, Jingjing Huang, Holly Krieger,
  Diego Marques, Jonathan Mason, Martin Mereb, and S.~Robert Wilson.
\newblock Some consequences of {S}chanuel's conjecture.
\newblock {\em J. Number Theory}, 129(6):1464--1467, 2009.

\bibitem[Ete18]{EterovicSchanuel}
Sebastian Eterovi\'{c}.
\newblock A {S}chanuel property for {$j$}.
\newblock {\em Bull. Lond. Math. Soc.}, 50(2):293--315, 2018.

\bibitem[Ete22]{gensol}
Sebastian Eterovi\'c.
\newblock Generic solutions of equations involving the modular $j$-function,
  2022.
\newblock Preprint: \url{https://arxiv.org/abs/2209.12192}.

\bibitem[Kir09]{kirby-semiab}
Jonathan Kirby.
\newblock The theory of the exponential differential equations of semiabelian
  varieties.
\newblock {\em Selecta Math. (N.S.)}, 15(3):445--486, 2009.

\bibitem[Kir10]{kirby-expalg}
Jonathan Kirby.
\newblock Exponential algebraicity in exponential fields.
\newblock {\em Bull. Lond. Math. Soc.}, 42(5):879--890, 2010.

\bibitem[Lan87]{lang-elliptic}
Serge Lang.
\newblock {\em Elliptic functions}, volume 112 of {\em Graduate Texts in
  Mathematics}.
\newblock Springer-Verlag, New York, second edition, 1987.
\newblock With an appendix by J. Tate.

\bibitem[Lan02]{Lang}
Serge Lang.
\newblock {\em Algebra}, volume 211 of {\em Graduate Texts in Mathematics}.
\newblock Springer-Verlag, New York, third edition, 2002.

\bibitem[Mas03]{masser2003heights}
David Masser.
\newblock Heights, transcendence, and linear independence on commutative group
  varieties.
\newblock In {\em Diophantine approximation}, pages 1--51. Springer, 2003.

\bibitem[MPT19]{mpt}
Ngaiming Mok, Jonathan Pila, and Jacob Tsimerman.
\newblock Ax-{S}chanuel for {S}himura varieties.
\newblock {\em Ann. of Math. (2)}, 189(3):945--978, 2019.

\bibitem[PSS20]{Phillipon}
Patrice Philippon, Biswajyoti Saha, and Ekata Saha.
\newblock An abelian analogue of {S}chanuel's conjecture and applications.
\newblock {\em Ramanujan J.}, 52(2):381--392, 2020.

\bibitem[PT16]{AxSchanuel4j}
Jonathan Pila and Jacob Tsimerman.
\newblock Ax-{S}chanuel for the {$j$}-function.
\newblock {\em Duke Math. J.}, 165(13):2587--2605, 2016.

\bibitem[Sch37]{schneider}
Theodor Schneider.
\newblock Arithmetische {U}ntersuchungen elliptischer {I}ntegrale.
\newblock {\em Math. Ann.}, 113(1):1--13, 1937.

\bibitem[TZ12]{tent-ziegler}
Katrin Tent and Martin Ziegler.
\newblock {\em A course in model theory}, volume~40 of {\em Lecture Notes in
  Logic}.
\newblock Association for Symbolic Logic, La Jolla, CA; Cambridge University
  Press, Cambridge, 2012.

\bibitem[Wil08]{Wilkie}
A.~J. Wilkie.
\newblock Some local definability theory for holomorphic functions.
\newblock In {\em Model theory with applications to algebra and analysis.
  {V}ol. 1}, volume 349 of {\em London Math. Soc. Lecture Note Ser.}, pages
  197--213. Cambridge Univ. Press, Cambridge, 2008.

\end{thebibliography}

\end{document}